\documentclass[12pt]{article}
\usepackage{mathrsfs}
\usepackage{amssymb}
\usepackage[hang,flushmargin]{footmisc}
\usepackage{multirow}
\usepackage{tikz}
\usetikzlibrary{chains,scopes,positioning,backgrounds,shapes,fit,shadows,calc}
\usepackage{amsfonts}
\usepackage{graphicx}
\usepackage{subfigure}
\usepackage{amsmath,amsthm,amssymb,latexsym,mathrsfs,amsfonts,amssymb}
\usepackage[colorlinks, linkcolor=blue, anchorcolor=gree, citecolor=red]{hyperref}
\usepackage[numbers,sort&compress]{natbib}
\usepackage{xcolor}
\usepackage{enumerate}
\usepackage{bm}
 \makeatletter
    
    \newcommand{\Rmnum}[1]
    {\expandafter\@slowromancap\romannumeral #1@}
    \makeatother
\newtheorem{thm}{Theorem}[section]

\newtheorem{lemma}[thm]{Lemma}
\newcounter{foo}[subsection]

\newtheorem{Remark}[foo]{Remark}

\newcounter{fooo}[section]

\newtheorem{stepp}[fooo]{Step}

\newtheorem{con}[thm]{Conjecture}

\newtheorem{example}[thm]{Example}
\newtheorem{defin}[thm]{Definition}

\theoremstyle{definition}
\newtheorem{remark}[thm]{Remark}

\usepackage{CJK}
\usepackage{caption}
\begin{document}

\renewcommand{\baselinestretch}{1.3}
\title{On Fan's conjecture about $4$-flow}

\author{Deping Song\textsuperscript{a}\quad Shuang Li\textsuperscript{a}\quad Xiao Wang\textsuperscript{b}
\\
\\{\footnotesize  \textsuperscript{a} \em Laboratory  of Mathematics and Complex System $($Ministry of Education$)$,}\\{\footnotesize \em School of Mathematical Sciences, Beijing Normal University, Beijing, 100875, China } 
\\{\footnotesize  \textsuperscript{b} \em College of Mathematics and Computer Application, Shangluo University,
	Shangluo, 726000,  China } }

\date{}
\maketitle
\footnote{\scriptsize\qquad{\em E-mail address:} ~~202021130067@mail.bnu.edu.cn ~~(D. Song),  202121130071@mail.bnu.edu.cn ~~(S. Li),\\
~~wxiao@mail.nwpu.edu.cn ~~(X. Wang).}

\begin{abstract}

Let $G$ be a bridgeless graph,~$C$ is a circuit of $G$.~Fan proposed a conjecture that if $G/C$ admits a nowhere-zero 4-flow,~then $G$ admits a 4-flow $(D,f)$ such that $E(G)-E(C)\subseteq$ supp$(f)$ and $|\textrm{supp}(f)\cap E(C)|>\frac{3}{4}|E(C)|$.~The purpose of this conjecture is to find shorter circuit cover in bridgeless graphs. Fan showed that the conjecture holds for $|E(C)|\le19.$ Wang,~Lu and Zhang showed that the conjecture holds for $|E(C)|\le 27$.~In this paper,~we prove that the conjecture holds for $|E(C)|\le 35.$

\end{abstract}

\section{Introduction}

Graphs considered in this paper may have loops and multiple edges.~For terminology and notations not defined here,~we follow~\cite{Bondy,Zhang_1}.
A \it circuit \rm is a  $2$-regular connected graph. A \it cycle \rm is a   graph such that the degree of each vertex is even.~\it Contracting \rm an edge means deleting the edge and then identifying its ends. For a subgraph $H$~of a graph $G$,~let $G/H$ be the graph obtained from $G$ by contracting all edges of $H$. A \it weight \rm of a graph $G$ is a function $\omega : E(G)\to \mathbb{Z}^+$,~where $\mathbb{Z}^+$ is the set of positive integers.
For a sugraph $H$ of $G$,~denote
$$ \omega (H)=\omega(E(H))=\sum\limits_{e\in E(H)}\omega(e).$$

Let $G$ be a graph,~$(D,f)$~be an ordered pair where $D$ is an orientation of $E(G)$~and $f:E(G)\to \Gamma$~is a function   where $\Gamma$ is an abelian group (an additive group with ``$0$" as its identity).~An oriented edge of $G$ (under the orientation $D$)~is called an \it arc. \rm The graph $G$ under the orientation $D$ is denoted by $D(G)$.~\rm For a vertex $v\in V(G)$,~let $E^+(v)$~(or $E^-(v)$)~be the set of all arcs of $D(G)$ with their tails~(or,~heads, respectively)~at the vertex $v$ and let $$f^+(v)=\sum\limits_{e\in E^+(v)}f(e)$$
and
$$f^-(v)=\sum\limits_{e\in E^-(v)}f(e).$$
 
A \it flow \rm of a graph $G$ is an ordered pair $(D,f)$ such that $$f^+(v)=f^-(v)$$ for every vertex $v\in V(G).$ An \it integer flow \rm $(D,f)$ is a flow of $G$ with an integer-valued function $f$.~An \it integer $k$-flow \rm of $G$ is an integer flow $(D,f)$ such that $|f(e)|<k$ for every edge $e\in G$.~The \it support \rm of $f$ is the set of all edges of $G$ with $f(e)\ne 0$ and is denoted by $supp(f)$. A flow $(D,f)$ of a graph $G$ is \it nowhere-zero \rm if $supp(f)=E(G)$.

For $\mathbb{Z}_2\times \mathbb{Z}_2$-flow~$(D,f)$,~it is not difficult to check that for each direction~$D'$,~$(D',f)$~is a~$\mathbb{Z}_2\times \mathbb{Z}_2$-flow.~For convenience,~$(D,f)$~is denoted by $f$.~If $C$ is a cycle of a graph $G$ and $a\in \mathbb{Z}_2\times \mathbb{Z}_2-\left\lbrace (0,0)\right\rbrace $,~Define $f_C$,~$f_{C,a}$~as follows:

\begin{equation}
	f_C(e)=\left\lbrace 
	\begin{array}{lllllllll}
		a&&&\rm if~\it e\in E(C);\\
		(0,0)&&&\rm if~\it e\in E(G)\backslash E(C) .
	\end{array}
	\nonumber
	\right.
\end{equation}

Then $f_{C,a}$ is a $\mathbb{Z}_2\times \mathbb{Z}_2$-flow of $G$,~$f_{C,(1,1)}$ is denoted by $f_{C}$.

In 2017,~in order to find shorter circuit cover of graphs,~Fan $($\cite{Fan}$)$~proposed the following conjecture.

\begin{con}\rm(\cite{Fan})\label{1.1}
	\it Let $G$ be a graph with a circuit $C$.~If $G/C$ admits a nowhere-zero $4$-flow,~then $G$ admits a $4$-flow $(D,f)$ such that $E(G)- E(C)\subseteq supp(f)$ and $ | E_{g=0}(C)|<\frac{1}{4}|E(C)|.$
\end{con}

For this conjecture,~Fan(\cite{Fan}) proved that Conjecture is true for~$|E(C)|\le 19$.~As an application of this result,~Fan proved that if $G$ is a bridgeless graph with minimum degree at least three,~then $cc(G)< \frac{278}{171}|E(G)|\approx 1.62573|E(G)|$,~\\which improved the result $cc(G)< \frac{44}{27}|E(G)|\approx1.62963|E(G)|$ by Kaiser et al~\cite{K}.~This conjecture can be refined as follows:

\begin{con}\label{Conjecture 1.2}
	\it Let $G$ be a graph with a circuit $C$ and $\omega:E(G)\to \mathbb{Z}$.~If there is a $\mathbb{Z}_2\times \mathbb{Z}_2$-flow  $f$~such that $E(G)-E(C)\subseteq supp(f)$,~then there is a $\mathbb{Z}_2\times \mathbb{Z}_2$-flow  $g$~in $G$ such that $E(G)-E(C)\subseteq supp(g)$ and $\omega(E_{g=(0,0)}(C))< \frac{1}{4}\omega(C)$.	
\end{con}

Recently,~Wang,~Lu and Zhang(\cite{WLG}) proved that Conjecture \ref{Conjecture 1.2} is true for $\omega(C)\le 27$.~As an application of this result,~they proved that if $G$ is a bridgeless graph with minimum degree at least three,~then $cc(G)<\frac{394}{243}|E(G)|\approx 1.6214|E(G)|$;~if $G$ is a bridgeless and loopless graph   with minimum degree at least three,~then $cc(G)< \frac{334}{207}|E(G)|\approx 1.6135|E(G)|$.

We will prove that Conjecture \ref{Conjecture 1.2} is true for $\omega(C)\le 35$.

\begin{thm}\label{1.3}
	\it Let $G$ be a graph with a circuit $C$ and $\omega:E(G)\to \mathbb{Z}$.~If~$ \omega(C)\le 35$~and there is a $\mathbb{Z}_2\times \mathbb{Z}_2$-flow  $f$~in $G$~such that $E(G)- E(C)\subseteq supp(f)$,~then there is a $\mathbb{Z}_2\times \mathbb{Z}_2$-flow  $g$~in $G$ such that $E(G)- E(C)\subseteq supp(g)$ and $\omega(E_{g=(0,0)}(C))$~\\ $< \frac{1}{4}\omega(C)$.
	
\end{thm}

Theorem $\ref{1.3}$ can be restated as follows.
\begin{thm}
	\it Let $G$ be a graph with a circuit $C$.~If~$ |E(C)|\le 35$~and there is a $4$-flow  $(D,f)$~in $G$~such that $E(G)- E(C)\subseteq supp(f)$,~then there is a $4$-flow  $(D,g)$~in $G$ such that $E(G)- E(C)\subseteq supp(g)$ and $|E_{g=0}(C)|$ $< \frac{1}{4}|E(C)|$.
\end{thm}

As an application of this result,~one can use the similar way in \cite{Fan} to check that if $G$ is a bridgeless graph   with minimum degree at least three,~then $cc(G)< \frac{34}{21}|E(G)|\approx 1.6190|E(G)|$;~if $G$ is a bridgeless and loopless graph   with minimum degree at least three,~then $cc(G)< \frac{50}{31}|E(G)|\approx 1.6129|E(G)|$.

\section{Preliminaries and Lemmas}
Let $G$ be a graph and $H$ be a subgraph of $G$. For a $\mathbb{Z}_2\times\mathbb{Z}_2$-flow~$f$,~define an \it equivalent class \rm of $f$ with respect to $H$ by 
$$\mathcal{R}_H(f)=\left\lbrace g:g~\textrm{is a}~\mathbb{Z}_2\times\mathbb{Z}_2\textrm{-flow in $G$ and $supp(g)-E(H)=supp(f)-E(H)$}\right\rbrace .$$

And for $a\in\mathbb{Z}_2\times\mathbb{Z}_2$,~let $$ E_{f=a}(H)=\left\lbrace e\in E(H):f(e)=a \right\rbrace. $$

Let $f$ be a $\mathbb{Z}_2\times \mathbb{Z}_2$-flow in $G$~and~$\omega:E(G)\to \mathbb{Z}^+,$~$ C$~is a circuit of $G$.~For a $\mathbb{Z}_2\times \mathbb{Z}_2$-flow $g\in \mathcal{R}_C(f)$~and $xy,yz\in E(C)$ with $g(xy)=g(yz)$,~we denote by $G^*$ the graph obtained from $G$ by \it lifting \rm $xy,yz\in E(C)$,~that is,~deleting $xy,yz$ and adding a new edge $e^*=xz$ and let

\begin{equation}
\omega^*(e)=\left\lbrace 
	\begin{array}{lllllllll}
		\omega(xy)+\omega(yz)&&&\rm if~\it e=e^*;\\
		\omega(e)&&&\rm if~\it e\in E(G^*)-\left\lbrace e^*\right\rbrace .
	\end{array}
	\nonumber
	\right.
\end{equation}
the resulted circuit is denoted by $C^*$.

Let 
\begin{equation}
	f^*(e)=\left\lbrace 
	\begin{array}{lllllllll}
		g(xy)&&&\rm if~\it e=e^*;\\
		g(e)&&&\rm if~\it e\in E(G^*)-\left\lbrace e^*\right\rbrace .
	\end{array}
	\nonumber
	\right.
\end{equation}

Then $f^*$~is also a $\mathbb{Z}_2\times \mathbb{Z}_2$-flow in $G$ and $supp(f^*)-E(C^*)=supp{f}-E(C)$.~In addition,~$|E(G^*)|<|E(G)|$.

Let $G$ be a graph admitting a $\mathbb{Z}_2\times\mathbb{Z}_2$-flow~$f$.~$C$~is a circuit of $G$.~A segment $S$ is $E_{f=(0,0)}$-$alternating$ in $C$ if the edges of $S$ are alternately in $E_{f=(0,0)}\cap E(C)$ and $E(C)-E_{f=(0,0)}$. An $E_{f=(0,0)}$-alternating segment \rm is \it maximal \rm if for any $g\in \mathcal{R}_C(f)$,~there is no $E_{g=(0,0)}$-$alternating$ segment~$S'$~such that
$$ E_{f=(0,0)}(S)\subseteq E_{g=(0,0)}(S')\textrm{~~~and~~~$|E(S)|<|E(S')|$.}$$

The following two lemmas are used in the proof of \rm Theorem $\ref{1.3}$.

\begin{lemma}\rm(\cite{T})
\it Let $\Gamma$ be an abelian group of order $k$. A graph $G$ admits a nowhere-zero $k$-flow if and only if $G$ admits a nowhere-zero $\Gamma$-flow.
\end{lemma}

The following lemma is similar to Lemma $3.2$ in \cite{WLG}.

\begin{lemma}\label{lem1}
Let $G$ be a graph~and~$\sigma$ be a permutation on $\left\lbrace (1,0),(0,1),(1,1)\right\rbrace $.~If there is a nowhere-zero~$\mathbb{Z}_2\times\mathbb{Z}_2 $-flow~$f$,~then $\sigma f$ is a nowhere-zero~$\mathbb{Z}_2\times\mathbb{Z}_2$-flow.
\end{lemma}

\proof 

By the definition of $\mathbb{Z}_2\times\mathbb{Z}_2$-flow,~we have $G-E_{g=a}(G)$~is a cycle,~then $E_{f=a}(G)$~is a partity subgraph of $G$.~Let $S_1=E_{f=\sigma^{-1}(1,0)}(G)\cup E_{f=\sigma^{-1}(1,1)}(G)$,~$S_2=E_{f=\sigma^{-1}(0,1)}(G)\cup E_{f=\sigma^{-1}(1,1)}(G)$,~then $\left\lbrace S_1,~S_2\right\rbrace$  is a cycle cover of $G$. Let 

\begin{equation}
	g(e)=\left\lbrace 
	\begin{array}{lllllllll}
		(1,0)&&&\rm if~\it e\in E(S_{\rm 1})-E(S_{\rm 2});\\
		(0,1)&&&\rm if~\it e\in E(S_{\rm 2})-E(S_{\rm 1});\\
		(1,1)&&&\rm if~\it e\in E(S_{\rm 1})\cap E(S_{\rm 2}).\\
	\end{array}
	\nonumber
	\right.
\end{equation}
We can verify that $g$ is a nowhere-zero~$\mathbb{Z}_2\times\mathbb{Z}_2$-flow of $G$ and $g=\sigma f$
$\hfill \qed$

\section{Maximal alternating segment of circuit $C$}

From now on,~we suppose there is a quadruple $(G, f, C,\omega )$ such that

(S1) $f$ is a $\mathbb{Z}_2\times \mathbb{Z}_2$-flow in $G$, $C$ is a circuit in $G$ and $\omega:E(G)\to \mathbb{Z}^+$ with $\omega(C)\le35$;

(S2) Subject to (S1), $\omega(E_{g=(0,0)}(C))\geq \frac{1}{4}\omega(C)$ for any $g\in \mathcal{R}_C(f)$;

(S3) Subject to (S1) and (S2), $\left|E(G)\right|$ is as small as possible.

For $\omega(C)\not\equiv 0 \left( \textrm{mod}4\right) $,~there is an $a\in \mathbb{Z}_2\times\mathbb{Z}_2$ such that $ E_{f=a}<\frac{1}{4}\omega(C)$.~If $a=(0,0)$,~we are done;~if $a\ne (0,0)$,~by Lemma \ref{lem1},~we can assume $a=(1,1)$,~let $g=f+f_C$,~then we are done.~So, from now on we always assume that $\omega(C)\le 32$ and  $\omega(C)\equiv 0 \left( \textrm{mod}4\right) $.

For $g\in \mathcal{R}_C(f) $ and $\left\lbrace a,b \right\rbrace \subset \left\lbrace (1,0),(0,1),(1,1)\right\rbrace $.~Let $P_{a,b}$ be the subgraph of $G$ induced by the edges in $E_{g=a}(G)\cup E_{g=b}(G)-E(C)$. Then the vertices with odd degree are in $V(C)$. Assume that $e_1$ and $e_2$ are two edges in $C$ with exactly one common end $x$ and $g(e_1)=a $,~$g(e_2)=b$.~We can obtain that $x$ is a vertex with odd degree in $ P_{a,a+b}$ immediately.~Let $y$ be another vertex with odd degree in the same component and $e_3$,~$e_4$ be two edges in $C$ with exactly one common end $y$.

Suppose $C=v_0v_1\cdots v_{|E(C)|-1}$,~$S=v_1\cdots v_{|E(S)|+1}$ and denote $v_{-1}=v_{|E(C)|-1}$.
By Lemma \ref*{lem1}, we can assume that $ g(v_0v_1)=(1,0)=a$,~$g(v_0v_{-1})=(1,1)=b$,~$g(v_{|E(S)|+1}v_{|E(S)|+2})=c$ and $g(v_{|E(S)|+2}v_{|E(S)|+3})=d$. Then $v_0$ is a vertex with odd degree in $ P_{a,a+b}$ and $v_{|E(S)|+1}$ is a vertex with odd degree in $ P_{c,c+d}$, assume $v_t$ be another vertex with odd degree in the same component of $ P_{a,a+b}$ and $v_{\bar{t}}$ be another vertex with odd degree in the same component of $ P_{c,c+d}$.~$v_1$ is a vertex with odd degree in $ P_{a,b}$,~assume $v_{t'}$ be another vertex with odd degree in the same component of $ P_{a,b}$,~moreover,~if $t'\ge |E(S)|+1$,~then there is a circuit $ C_1$~consist of a path $P_1$ from $v_1$ to $v_{t'}$ and  $v_0v_{-1}\cdots v_{t'}$.~Let $g'=g+f_{C_1,(a+b)}$,~then $v_0$ is a vertex with odd degree in $ P_{b,a+b}$ in $g'$, assume that $v_{t''}$ is another vertex with odd degree in the same component of $ P_{b,a+b}$ in $g'$.

Then we have the following lemmas.

\begin{lemma}\label{l1}
	$\omega(E_{g=a}(C))=\frac{\omega(C)}{4}\le 8$ for any $g\in \mathcal{R}_C(f)$ and each $a\in \mathbb{Z}_2\times \mathbb{Z}_2$.	
\end{lemma}

\begin{proof}
	Suppose to be contrary that there is a flow $g\in \mathcal{R}_C(f)$ and $a\in \mathbb{Z}_2\times \mathbb{Z}_2$ such that $\omega(E_{g=a}(C))< \frac{\omega(C)}{4}$. 
	By (S2), $\omega(E_{g=(0,0)}(C))\ge\frac{\omega(C)}{4}$. By Lemma \ref{lem1}, we can assume that $\omega(E_{g=(1,1)}(C))<\frac{\omega(C)}{4}$. Let $g'=g+f_C$, then $(G, g', C, \omega)$ is a new quadruple satisfying (S1). Moreover, $g'\in \mathcal{R}_C(f)$ and $$\omega(E_{g'=(0,0)}(C))<\frac{\omega(C)}{4}.$$
	This contradicts (S2). 
\end{proof}

\begin{lemma}\label{l2}
	$E_{g=a}(C)$ is a matching for any $g\in \mathcal{R}_C(f)$ and each $a\in \mathbb{Z}_2\times \mathbb{Z}_2$.
\end{lemma}

\begin{proof}
	Suppose to be contrary that there is a function $g\in \mathcal{R}_C(f)$ and $a\in \mathbb{Z}_2\times \mathbb{Z}_2$ such that $E_{g=a}$ is not a matching. Then there are two edges $xy, yz\in E(C)$ such that $g(xy)=g(yz)=a$. By \it lifting \rm $xy$ and $yz$, we can obtain a new quadruple $(G, f^*, C^*,\omega ^*)$ which satisfies (S1) and (S2), however, $\left|E(G^*)\right|=\left|E(G)\right|-1,$ which is a contradiction as condition (S3).
\end{proof}

\begin{lemma}\label{l3}
	Let $e_1$ and $e_2$ be two edges in $C$ with exactly one common end $x$ and $g(e_1)=a $,~$g(e_2)=b$,~$y$ be another vertex with odd degree in the same component of $ P_{a,a+b}$ and $e_3$,~$e_4$ be two edges in $C$ with exactly one common end $y$. Then $g(e_3)+g(e_4)\ne b$ Let $Q$ be a path of $C$ from $x$ to $y$,~then for each $c\in  \mathbb{Z}_2\times \mathbb{Z}_2$,~then $ \omega(E_{g=c}(Q))=\omega(E_{g=c+b}(Q))$.
\end{lemma}

\begin{proof}
	If $g(e_3)+g(e_4)= b$, then $g'(e_3)=g'(e_4)$, which is a contradiction as Lemma \ref{l2}. Thus, we have $g(e_3)+g(e_4)\ne b$.

If $\omega(E_{g=c}(Q))\ne\omega(E_{g=c+b}(Q))$, then $ \omega(E_{g=c}(Q))\ne \frac{\omega(C)}{4}$, which is a contradiction as Lemma $\ref{l1}$. Thus, we have $\omega(E_{g=c}(Q))=\omega(E_{g=c+b}(Q))$.
\end{proof}

\begin{lemma}\label{l6}
	Let $P_1=v_1v_2\cdots v_r$,~$P_2=v_{s}v_{s+1}\cdots v_{t}$ be two paths of $C$ $(r\ge s+1)$.~Let $Q_1$~be a path in $ P_{a,a+b}$ from $v_1$ to $v_r$ and $Q_2$~be a path in $ P_{a,a+b}$ from $v_s$ to $v_t$,~then $$ \omega(E_{g=c}(v_1\cdots v_{t}))=\omega(E_{g=c+b}(v_1\cdots v_{s}))$$
\end{lemma}

\begin{proof}
By Lemma \ref{l3},~$ \omega(E_{g=c}(v_1\cdots v_{t}))=\omega(E_{g=c+b}(v_1\cdots v_{t}))$.~Let $C_i=P_i\cup Q_i$,~$i=1,2$. We can check that $g_1=g+f_{C_1,b}+f_{C_2,b}\in \mathcal{R}_C(f)$,~then $\omega(E_{g=c}(v_1\cdots v_{s}))+\omega(E_{g=c}(v_r\cdots v_{t}))=\omega(E_{g=c+b}(v_1\cdots v_{s}))+\omega(E_{g=c+b}(v_r\cdots v_{t}))$,
we have $$\omega(E_{g=c}(v_1\cdots v_{t}))=\omega(E_{g=c+b}(v_1\cdots v_{s}))$$
\end{proof}
Let $\Omega$ be the maximum weight of $C$.

\begin{lemma}\label{l4}
	Let $g\in \mathcal{R}_C(f)$,~$S$~be a maximal $E_{g=(0,0)}$-alternating segment of $C$. If there is an edge $e\in E_{g=(0,0)}(S)$ with $\omega(e)=\Omega$,~then $|E(S)|\le |E(C)|-2$.~If~$\Omega= 3$ and there is an edge $e\in E_{f=(0,0)}(S)$ such that $\omega(e)=2$,~then $|E(S)|\le |E(C)|-2$.
\end{lemma}
\proof Since $E_{g=\left( 0,0\right) }(S)$ is a matching by Lemma $\ref{l2}$~and the definition of $S$,~$|E(S)|\le |E(C)|-1$.

Suppose that $|E(S)|=|E(C)|-1$.~If there is an edge $e\in E_{g=(0,0)}(S)$ with $\omega(e)=\Omega$,~then~

\begin{equation}
	\begin{array}{lllllllllll}
		0=\omega(C)-\omega(C)&\le |E_{f=(0,0)}\cap E(C)|+\Omega|E(C)-E_{f=(0,0)}|-\omega(C)\\
		&\le \frac{\omega(C)}{4}+\Omega(\frac{\omega(C)}{4}+1-\Omega)-\omega(C)\\
		&\le \frac{1}{64}\left(\omega(C) -20\right)^2-6<0,
	\end{array}
	\nonumber
\end{equation}
a contradiction.~If~$\Omega= 3$ and there is an edge $e\in E_{f=(0,0)}(S)$ such that $\omega(e)=2$,~then \begin{equation}
	\begin{array}{lllllllllll}
		0=\omega(C)-\omega(C)&\le |E_{f=(0,0)}\cap E(C)|+\Omega|E(C)-E_{f=(0,0)}|-\omega(C)\\
		&\le \frac{\omega(C)}{4}+\Omega(\frac{\omega(C)}{4}+2-\Omega)-\omega(C)\\
		&=-3<0,
	\end{array}
	\nonumber
\end{equation} a contradiction.
$\hfill \qed$

\begin{lemma}\label{l5}
	Let $g\in \mathcal{R}_C(f)$,~$S$~be a maximal $E_{g=(0,0)}$-alternating segment of $C$, then $5\le t\le |E(S)|$ and $t\ne 6$. If $t=7$, then $g(v_2v_3)=(1,1)$ and $g(v_4v_5)=(1,1)$. Furthermore, $t\ge 7$ if there is an edge in $ \left\lbrace v_1v_2,v_3v_4\right\rbrace $ weighted $\Omega$. 
\end{lemma}

\begin{proof}
	
	Let $P_1=v_0v_1v_2v_3$,~$P_2=v_3v_4v_5v_6$.~By Lemma \ref{l3},~$g(v_1v_2)=g(v_3v_4)=(0,0)$,~$g(v_0v_1)=a$,~we have $t\ge 5$ and $t\ge 7$ if there is an edge in $ \left\lbrace v_1v_2,v_3v_4\right\rbrace $ weighted $\Omega$. If $ t\ge|E(S)|+1$,~then there is a circuit $ C_1$~consist of a path $P_1$ from $v_1$ to $v_{t}$ and  $v_0v_{-1}\cdots v_{t}$.~Let $g'=g+f_{C_1,(b)}$, then $g'\in \mathcal{R}_C(f)$ and $ g'(v_0v_{-1})=g'(v_0v_1)=(0,0)$,~a contradiction as Lemma \ref{l2}.
	Let's prove that $t\ne 6$. Assume that $t=6$, we have $g(v_2v_3)=(0,1)$ or $g(v_4v_5)=(0,1)$. If $g(v_2v_3)=(0,1)$, then $g(v_4v_5)=(1,1)$. Consider $P_{(0,1),(1,0)}$, $v_{3}$ be a vertex with odd degree on $ P_{(0,1),(1,0)}$, let $v_{t_0}$ be another vertex with odd degree in the same component of $ P_{(0,1),(1,0)}$. ~If $v_{t_0}\notin \left\lbrace v_0,v_1,\dots,v_6\right\rbrace $, then $\omega(E_{g=(0,0)}(P_1))=\omega(E_{g=(1,1)}(P_1))$, which is a contradiction as $g(v_0v_1)=(1,0)$, $g(v_1v_2)=(0,0)$ and $g(v_2v_3)= (0,1)$. Similarily, we have $t_0\notin \left\lbrace 0,1\right\rbrace $, thus $g(v_2v_3)\ne (0,1).$ Similarily, we can prove that $g(v_4v_5)\ne (0,1)$. This is a contradiction.
	
	Finally, if $t=7$, let's prove that $g(v_2v_3)=(1,1)$ and $g(v_4v_5)=(1,1)$. Since $t=7$ and Lemma \ref{l3}, $g(v_6v_7)=(0,1)$ or $(1,0)$. If $g(v_2v_3)=(0,1)$ or $(1,0)$, then $g(v_4v_5)=(1,1)$. Consider $P_{(0,1),(1,0)}$, $v_{3}$ be a vertex with odd degree on $ P_{(0,1),(1,0)}$, let $v_{t^*}$ be another vertex with odd degree in the same component of $ P_{(0,1),(1,0)}$. If $v_{t^*}\notin \left\lbrace v_0,v_1,\dots,v_7\right\rbrace $, then $\omega(E_{g=(0,0)}(P_1))=\omega(E_{g=(1,1)}(P_1))$, which is a contradiction as $g(v_0v_1)=(1,0)$, $g(v_1v_2)=(0,0)$ and $g(v_2v_3)=(1,0)$ or $(0,1)$. Similarily, we have $t^*\notin \left\lbrace 0,1,7 \right\rbrace $, thus $t^*=6$. Since $\omega(v_4v_5)=\omega(v_1v_2)+\omega(v_3v_4)+\omega(v_5v_6)$, $\omega(v_4v_5)\ne \omega(v_3v_4)+\omega(v_5v_6)$, which is a contradiction as $\omega(E_{g=(0,0)}(P_2))=\omega(E_{g=(1,1)}(P_2))$. Thus, $g(v_2v_3)=(1,1)$. Similarily, we can prove that $g(v_4v_5)=(1,1)$.

\end{proof}

\begin{Remark}\label{re1}
		Let $g\in \mathcal{R}_C(f)$,~$S$~be a maximal $E_{g=(0,0)}$-alternating segment of $C$,~then by Lemma $\ref{l5}$,~$t''\le |E(S)|$ if $t''$ is defined.
\end{Remark}

Now,~adding a condiction,~we suppose that $S$ is a maximal $E_{g=(0,0)}$-alternating segment in $C$ such that there is an edge $e\in E(S)$ which is weighted $\Omega$ and $g(e)=(0,0)$.

\begin{lemma}\label{l41}
	$|E(S)|\ge 9$ and $\Omega\le4$.
\end{lemma}

\begin{proof}
	Suppose to be contrary that $|E(S)|<9$.~By symmetry,~assume there is an edge in $ \left\lbrace v_1v_2,v_3v_4\right\rbrace $ weighted $\Omega$.~By Lemma $\ref{l5}$,~$|E(S)|=7$ and $t=7$.~By Lemma \ref{l3} $g(v_3v_4)=g(v_5v_6)=(1,1)$,~$g(v_7v_8)=(0,1)$ and~$ t'\ge 8$.~Futhermore, $t''\ge 8$ by Lemma \ref{l3},~which is a contradiction as Lemma \ref{l5}. Since\begin{equation}
		\begin{array}{lllllllllll}
			35\ge\omega(C)&= 4|E_{g=(0,0)}\cap E(C)|\\
			&\ge 4(\Omega+4)\\
		\end{array}
		\nonumber
	\end{equation}
	
	Thus,~$\Omega\le4$.
	
\end{proof}

By symmetry,~assume there is an edge in $ \left\lbrace v_1v_2,v_3v_4,v_5v_6\right\rbrace $ weighted $\Omega$ if $9\le|E(S)|\le 11$.~If $|E(S)|=9$,~let the number of edges in $\left\lbrace v_1v_2,v_3v_4\right\rbrace$ as small as possible.

\begin{lemma}\label{l42}
	If~ $\Omega=4$, then $S$ satisfy the following two conditions:
	
	(a) $\omega(v_{5}v_{6})=4$ and $\omega(e)=1$ for any $e\in E_{g=(0,0)}(S)\backslash\left\lbrace v_{5}v_{6}\right\rbrace $;
	
	(b) $\omega(v_{2}v_{3})=\omega(v_{8}v_{9})=2$.
\end{lemma}

\begin{proof}
	Combining $\Omega=4$ with Lemma \ref{l41}, $\left|E(S)\right|= 9$ and $\omega(C)=32$,~there are four edges weighted $1$ in $E_{g=(0,0)}(S)$ and there is one edge  weighted $4$ in $E_{g=(0,0)}(S)$. By Lemma \ref{l5}, $7\leq t\leq \left|E(S)\right|=9$ or $t=5$.
	
	If $t\ge7$, then $g(v_2v_3)=(1,1)$  by the same proof of Lemma \ref{l5},~then $t'\ge 10$,~then $t''\ge 10$ a contradiction as Remark \ref*{re1}.
	
	So $t=5$. Combining with Lemma \ref{l3}, $\omega(v_{5}v_{6})=4$, $g(v_{2}v_{3})=(1,1)$, $\omega(v_{2}v_{3})=2$, $g(v_{4}v_{5})=(0,1)$ and $\omega(v_{4}v_{5})=\omega(v_{0}v_{1})$.
	
	By the definition of $\bar{t}$,~similarly,~we can obtain $\bar{t}=v_{6}$, $\omega(v_{8}v_{9})=2$ and $\omega(v_{6}v_{7})=\omega(v_{10}v_{11})$.
\end{proof}

\begin{lemma}\label{l43}
	If~ $\Omega=3$, then $S$ satisfy the following two conditions:

	(a) $\omega(v_{5}v_{6})=3$;
	
	(b) $\omega(v_{2}v_{3})=2$ and $\omega(v_{1}v_{2})=\omega(v_{3}v_{4})=1$.

\end{lemma}
\begin{proof}
	Combining $\Omega=3$ with Lemma \ref{l41}, $\left|E(S)\right|= 9$ or $|E(S)|=11$.~By Lemma \ref{l5}, $7\leq t\leq \left|E(S)\right|$ or $t=5$.
	
	For $|E(S)|=9$,~if $t\ge7$, then $g(v_2v_3)=(1,1)$. By the same proof of Lemma \ref{l5},~then $t'\ge 10$,~then $t''\ge 10$,~which is a contradiction as Lemma \ref{l5}.~Thus $t=5$ and there are at least three edges weighted $1$ in $E_{g=(0,0)}(S)$ by Lemma $\ref{l1}$. Therefore,~ $\omega(v_{2}v_{3})=2$ and $\omega(v_{1}v_{2})=\omega(v_{3}v_{4})=1$; or $\omega(v_{8}v_{9})=2$ and $\omega(v_{7}v_{8})=\omega(v_{9}v_{10})=1$;~by our assumption,~$\omega(v_{2}v_{3})=2$ and $\omega(v_{1}v_{2})=\omega(v_{3}v_{4})=1$.
	
	For $|E(S)|=11$, if $t\ge7$, then $g(v_2v_3)=(1,1)$ by the same proof of Lemma \ref{l5},~then $t''\ge 12$,~a contradiction as Lemma \ref{l5}. Thus $t=5$. Therefore,~ $\omega(v_{2}v_{3})=2$ and $\omega(v_{1}v_{2})=\omega(v_{3}v_{4})=1$.
\end{proof}

Let $S'=v_{1}'v_{2}'\dots v_{\left|E(S')\right|+1}'$ be a maximal $E_{g'=(0,0)}(G)$-alternating containing $v_{2}v_{3}$ and $g'(v_{2}v_{3})=(0,0)$ and $\omega(v_{2}v_{3})=2$ on $C$, $v_{2}v_{3}$ is shown in Figure \ref*{fig1} and \ref*{fig6}. 
Futhermore, if $\Omega=2$, let $S'=S$.
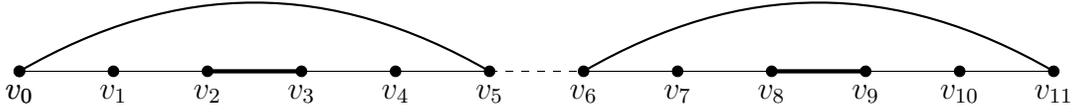
\begin{figure*}
	\begin{tikzpicture}
		\draw (0,0)--(6.25,0) ;
		\draw (7.5,0)--(13.75,0) ;
		\draw [color=white, fill=white,](0,-0.3) node [black]{$v_0$};
		\draw [color=white, fill=white,](1.25,-0.3) node [black]{$v_1$};
		\draw [color=white, fill=white,](2.5,-0.3) node [black]{$v_2$};
		\draw [color=white, fill=white,](3.75,-0.3) node [black]{$v_3$};
		\draw [color=white, fill=white,](5,-0.3) node [black]{$v_4$};
		\draw [color=white, fill=white,](6.25,-0.3) node [black]{$v_5$};
		\draw [color=white, fill=white,](7.5,-0.3) node [black]{$v_6$};
		\draw [color=white, fill=white,](8.75,-0.3) node [black]{$v_7$};
		\draw [color=white, fill=white,](10,-0.3) node [black]{$v_8$};
		\draw [color=white, fill=white,](11.25,-0.3) node [black]{$v_9$};
		\draw [color=white, fill=white,](12.5,-0.3) node [black]{$v_{10}$};
		\draw [color=white, fill=white,](0,-0.3) node [black]{$v_0$};
		\draw [color=white, fill=white,](13.75,-0.3) node [black]{$v_{11}$};
		\filldraw[color=black, fill=black,] (0,0) circle (2pt);
		\filldraw[color=black, fill=black,] (0,0) circle (2pt);
		\filldraw[color=black, fill=black,] (1.25,0) circle (2pt);
		\filldraw[color=black, fill=black,] (2.5,0) circle (2pt);
		\filldraw[color=black, fill=black,] (3.75,0) circle (2pt);
		\filldraw[color=black, fill=black,] (5,0) circle (2pt);
		\filldraw[color=black, fill=black,] (6.25,0) circle (2pt);
		\filldraw[color=black, fill=black,] (7.5,0) circle (2pt);
		\filldraw[color=black, fill=black,] (8.75,0) circle (2pt);
		\filldraw[color=black, fill=black,] (10,0) circle (2pt);
		\filldraw[color=black, fill=black,] (11.25,0) circle (2pt);
		\filldraw[color=black, fill=black,] (12.5,0) circle (2pt);
		\filldraw[color=black, fill=black,] (13.75,0) circle (2pt);
		\draw[-,thick] (0,0) to [in = 150, out =30](6.25,0);
		\draw[-,thick] (7.5,0) to [in = 150, out =30](13.75,0);
		\draw[ultra thick] (2.5,0)--(3.75,0);
		\draw[ultra thick] (10,0)--(11.25,0);
		\draw[dashed] (6.25,0)--(7.5,0);
	\end{tikzpicture}~~~~~~~~~~~~~~~~~~~~~~~~~~~~~
	\caption{$\omega(v_5v_6)=4$,~$\omega_(v_2v_3)=\omega_(v_8v_9)=2$,~$\omega_(v_1v_2)=\omega_(v_3v_4)=1$~and~\\$\text{~~~~~~~~~~~~~}\omega_(v_7v_8)=\omega_(v_9v_{10})=1$}
	\label{fig1}
\end{figure*}

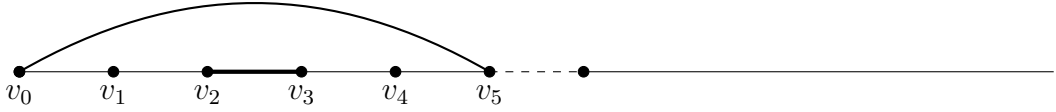
\begin{figure*}
	\begin{tikzpicture}
		\draw (0,0)--(6.25,0) ;
		\draw (7.5,0)--(13.75,0) ;
		\draw [color=white, fill=white,](0,-0.3) node [black]{$v_0$};
		\draw [color=white, fill=white,](1.25,-0.3) node [black]{$v_1$};
		\draw [color=white, fill=white,](2.5,-0.3) node [black]{$v_2$};
		\draw [color=white, fill=white,](3.75,-0.3) node [black]{$v_3$};
		\draw [color=white, fill=white,](5,-0.3) node [black]{$v_4$};
		\draw [color=white, fill=white,](6.25,-0.3) node [black]{$v_5$};
		\filldraw[color=black, fill=black,] (0,0) circle (2pt);
		\filldraw[color=black, fill=black,] (0,0) circle (2pt);
		\filldraw[color=black, fill=black,] (1.25,0) circle (2pt);
		\filldraw[color=black, fill=black,] (2.5,0) circle (2pt);
		\filldraw[color=black, fill=black,] (3.75,0) circle (2pt);
		\filldraw[color=black, fill=black,] (5,0) circle (2pt);
		\filldraw[color=black, fill=black,] (6.25,0) circle (2pt);
		\filldraw[color=black, fill=black,] (7.5,0) circle (2pt);
		\draw[-,thick] (0,0) to [in = 150, out =30](6.25,0);
		\draw[ultra thick] (2.5,0)--(3.75,0);
		\draw[dashed] (6.25,0)--(7.5,0);
	\end{tikzpicture}~~~~~~~~~~~~~~~~~~~~~~~~~~~~~
	\caption{$\omega(v_5v_6)=3$,~$\omega_(v_2v_3)=2$,~$\omega_(v_1v_2)=\omega_(v_3v_4)=1$}
	\label{fig6}
\end{figure*}

\begin{lemma}\label{l44}
	$|E(S')|\le |E(C)|-2$.

\end{lemma}
\begin{proof}
	Suppose to be contrary that $|E(S')|> |E(C)|-2$. Since $E_{g'=\left( 0,0\right) }(S')$ is a matching and $S\subset C$ , we have $|E(S')|= |E(C)|-1$. By Lemma \ref{l4}, $\Omega=4$. Let $S''$ be a maximal $E_{g''=(0,0)}(G)$-alternating of $C$ containing $e$ with $g''(e)=(0,0)$ and $\omega(e)=4$. By Lemma \ref{l41}, $|E(S'')|\geq 9$, then $\omega(C)=32$. Since $S'$ contains an edge $e_1$ with $g'(e_1)=(0,0)$ and $\omega(e_1)=2$, $|E(S')|\leq 13$. Combining with Lemma \ref{l42}, $g'(e)\ne (0,0)$. Since $\omega(e)=4$, there are two edges weighted 1.
	Then
	\begin{equation}
		\begin{array}{lllllllllll}
			0=\omega(C)-\omega(C)&= |E_{g'=(0,0)}\cap E(C)|+|E(C)-E_{g'=(0,0)}|-\omega(C)\\
			&\le \frac{\omega(C)}{4}+\Omega+1+1+\Omega(7-1-2)-\omega(C)\\
			&=-2<0,
		\end{array}
		\nonumber
	\end{equation} a contradiction.

\end{proof}

\begin{lemma}\label{l45}
	$|E(S')|\geq 9$.
\end{lemma}
\begin{proof}
	Suppose to be contrary that $|E(S')|=5$ or $7$. 
	
	If $|E(S')|=5$, then $t=5$. As shown in Figure 
	Consider $P_{(1,1),(1,0)}$, $v_{1}'$ be a vertex with odd degree on $ P_{(1,1),(1,0)}$, let $v_{t'}'$ be another vertex with odd degree in the same component of $ P_{(1,1),(1,0)}$. By Lemma \ref{1.3}, $t'\geq 6$. Combining the definition of $t''$ and Figure 
	, $t''\geq 6$,  which is a contradiction as Remark \ref{re1}. 
	
	If $|E(S')|=7$, by symmetry, we can assume that exist edge $e\in \left\lbrace v_{1}'v_{2}', v_{3}'v_{4}'\right\rbrace$ such that $\omega(e)=2$,~especially,~assume $e=v_3v_4$~for $\Omega\ne2$. By Lemma \ref{l5}, $t=5$ or $7$. If $t=5$, $g'(v_{2}'v_{3}')=(1,1)$, combining with $\omega(e)=2$, $\Omega\geq 3$. By Figure \ref{fig1}
	and   Figure  \ref{fig6}.
	, $\omega(v_{2}' v_{3}')=1$, which is a contradiction as Lemma \ref{l3}. If $t=7$, by Lemma \ref{l5}, $g'(v_{2}'v_{3}')=g'(v_{4}'v_{5}')=(1,1)$ and $g'(v_{6}'v_{7}')=(0,1)$. Consider $P_{(1,1),(1,0)}$, $v_{1}'$ be a vertex with odd degree on $ P_{(1,1),(1,0)}$, let $v_{t'}'$ be another vertex with odd degree in the same component of $ P_{(1,1),(1,0)}$. Combining with Lemma \ref{l3}, $t'\geq 8$. Since the definition of $t''$ and Lemma \ref{l3}
	$t''\geq 8$,  which is a contradiction as Remark \ref{re1}. Thus, $|E(S')|\geq 9$.
	
\end{proof}

\section{Proof of Theorem \ref{1.3}}

We still use the same notation as section $3$ and we suppose that $S$ is a maximal $E_{g=(0,0)}$-alternating segment in $C$ such that there is an edge $e\in E(S)$ which is weighted $\Omega$ and $g(e)=(0,0)$.~We only need to proof such $S'$ does not exist.~If there is an $S'$,~then $9\le|E(S')|\le13$ by Lemma \ref{l1} and Lemma \ref{l45}.

\begin{stepp}\label{step1}
	$\left|E(S')\right|\ne 9$.
\end{stepp}
Suppose to be contrary that $|E(S')|=9$. By symmetry, we can assume that exist edge $e\in \left\lbrace v_{1}'v_{2}', v_{3}'v_{4}', v_{5}'v_{6}'\right\rbrace$ such that $\omega(e)=2$,~especially,~assume $e=v_3v_4$~for $\Omega\ne2$.

If $t=5$, then $g'(v_{2}'v_{3}')=(1,1)$ and $g'(v_{4}'v_{5}')=(0,1)$. If $\Omega\geq 3$, combining \ref*{fig1}
and   \ref*{fig6}, $\omega(v_{1}'v_{2}')\ne 2$ and $\omega(v_{3}'v_{4}')\ne 2$. Thus, $\omega(v_{5}'v_{6}')=2$. Futhermore, $\omega(v_{4}'v_{5}')=\omega(v_{6}'v_{7})=1$. Consider $P_{(1,1),(1,0)}$, $v_{1}'$ be a vertex with odd degree on $ P_{(1,1),(1,0)}$, let $v_{t'}'$ be another vertex with odd degree in the same component of $ P_{(1,1),(1,0)}$. Combining with Lemma \ref{l3}, $t'\geq 10$. Since the definition of $t''$ and Lemma \ref{l3},~
$t''\geq 10$,  which is a contradiction as Remark  \ref{re1}. If $\Omega= 2$, similarily, we have $t''\geq 10$, which is a contradiction.

If $t=7$, by Lemma \ref{l5}, $g(v_{2}'v_{3}')=(1,1)$, $g(v_{4}'v_{5}')=(1,1)$ and $g(v_{6}'v_{7}')=(0,1)$. If $\Omega\leq 3$, then $t'\ge10$,~thus~$t''\geq 10$, which is a contradiction as Remark \ref{re1}. If $\Omega=4$, then $t'\geq 9$. If $t'=9$. By Figure 
$\omega(v_{1}'v_{2}')\ne 2$ and $\omega(v_{3}'v_{4}')\ne 2$. Thus, $\omega(v_{5}'v_{6}')=2$. Since $t'=9$, $\omega(E_{g'=(0,0)}(v_{1}'v_{2}'\dots v_{9}'))=\omega(E_{g'=(0,1)}(v_{1}'v_{2}'\dots v_{9}'))$. Thus, $\omega(v_{6}'v_{7}')=4$ and $\omega(v_{8}'v_{9}')=4$. By $\omega(v_{8}'v_{9}')=4$ and Lemma \ref{l43}, $\omega(v_{6}'v_{7}')=1$, which is a contradiction.

If $t=8$, then $\omega(E_{g'=(0,0)}(v_{0}'v_{1}'\dots v_{8}'))\geq 5$. Thus $\left| E_{g'=(1,1)}(v_{0}'v_{1}'\dots v_{8}')\right| \geq 2$. Similar to Lemma \ref{l5}, we can prove that $g(v_{2}'v_{3}')=(1,1)$ and $g(v_{6}'v_{7}')=(1,1)$. Thus, $g(v_{4}'v_{5})=(0,1)$. Futhermore, we have $t'\geq 10$ and $t''\geq 10$,  which is a contradiction as Remark \ref{re1}. 

If $t=9$, then $g'(v_{8}'v_{9}')=(1,0)$ or $(0,1)$. Similarily, we can prove that $g(v_{2}'v_{3}')=(1,1)$ and $g(v_{6}'v_{7}')=(1,1)$.
Futhermore, we have $t'\geq 10$ and $t''\geq 10$,  which is a contradiction as Remark \ref{re1}. 

Thus, we have $\left|E(S')\right|\ne 9$.	

\begin{stepp}\label{step2}
	$\left|E(S')\right|\ne 11$.
\end{stepp}
Suppose to be contrary that $|E(S')|=11$. By symmetry, we can assume that exist edge $e\in \left\lbrace v_{1}'v_{2}', v_{3}'v_{4}', v_{5}'v_{6}'\right\rbrace$ such that $\omega(e)=2$,~especially,~assume $e=v_3v_4$~for $\Omega\ne2$.

If $t=5$, then $g'(v_{2}'v_{3}')=(1,1)$ and $g'(v_{4}'v_{5}')=(0,1)$. If $\Omega\geq 3$, combining Figure \ref*{fig1}
and   Figure  \ref*{fig6}, $\omega(v_{1}'v_{2}')\ne 2$ and $\omega(v_{3}'v_{4}')\ne 2$. Thus, $\omega(v_{5}'v_{6}')=2$. Futhermore, $\omega(v_{4}'v_{5}')=\omega(v_{6}'v_{7})=1$. Consider $P_{(1,1),(1,0)}$, $v_{1}'$ be a vertex with odd degree on $ P_{(1,1),(1,0)}$, let $v_{t}'$ be another vertex with odd degree in the same component of $ P_{(1,1),(1,0)}$. Combining with Lemma \ref{l3}, $t'\geq 11$. If $t'= 11$. Since $\omega(E_{g'=(0,0)}(v_{1}'v_{2}'\dots v_{11}'))\geq 6$ and Lemma \ref{l3}, $\omega(E_{g'=(0,1)}(v_{1}'v_{2}'\dots v_{11}'))\geq 6$.
By Lemma \ref{l3}, $g'(v_{10}'v_{11}')\ne (0,1)$. Since $g'(v_{4}'v_{5}')=(0,1)$ and $\omega(v_{4}'v_{5}')=1$, $g'(v_{6}'v_{7}')=g'(v_{8}'v_{9}')=(0,1)$. By $\omega(v_{6}'v_{7}')=1$, $\omega(v_{8}'v_{9}')=4$. Combining with Lemma \ref{l43}, $\omega(v_{10}'v_{11}')=1$. Thus, $\omega(E_{g'=(1,1)}(v_{1}'v_{2}'\dots v_{11}'))\ne \omega(E_{g'=(1,0)}(v_{1}'v_{2}'\dots v_{11}'))$, which is a contradiction. Finally, we have $t'\geq 12$. Similarily, we can obtain that $t''\geq 12$, which is a contradiction as Remark \ref{re1}. If $\Omega=2$. By Lemma \ref{l3}, $t'\geq 11$. If $t'= 11$. Since 
$\omega(E_{g'=(0,0)}(v_{1}'v_{2}'\dots v_{11}'))\geq 6$, $\Omega=2$ and $g'(v_{10}'v_{11}')\ne (0,1)$, $g'(e)=(0,1)$ and $\omega(e)=2$ for any $e\in \left\lbrace v_{4}'v_{5}', v_{6}'v_{7}', v_{8}'v_{9}'\right\rbrace$. By $\omega(E_{g'=(0,0)}(v_{0}'v_{1}'\dots v_{5}'))\geq 2$ and Lemma \ref{l3}, $\omega(v_{2}'v_{3}')=2$. Thus, $\omega(v_{10}'v_{11}')=2$. Let $S^*$ is a maximal $E_{g^*=(0,0)}(G)$-alternating containing $v_{4}'v_{5}'\dots v_{9}'$ and $g^*(v_{4}'v_{5}')=g^*(v_{6}'v_{7}')=g^*(v_{8}'v_{9}')=(0,0)$ on $C$, since $\omega(E_g^*(0,0))\leq 8$, $\left| E(S^*)\right|= 7$, which is a contradiction as Lemma \ref{l41}. Thus, $t'\geq 12$. Similarily, $t''\geq 12$, which is a contradiction as Remark \ref{re1}.

If $t=7$, by Lemma \ref{l5}, $g'(v_{2}'v_{3}')=(1,1)$, $g(v_{4}'v_{5}')=(1,1)$ and $g'(v_{6}'v_{7}')=(0,1)$. If $\Omega\leq 3$. Consider $t'$, Since $\omega(v_{2}'v_{3}')+\omega(v_{4}'v_{5}')\geq 4$, $\omega(v_{1}'v_{2}')+\omega(v_{3}'v_{4}')+\omega(v_{5}'v_{6}')\ge 4$ and Lemma \ref{l3}, then $t'\geq 12$, futher, we have  $t''\geq 12$, which is a contradiction as Remark \ref{re1}. If $\Omega=4$, then $t'\geq 11$. If $t'=11$. Since $\omega(E_{g'=(0,0)}(v_{1}'v_{2}'\dots v_{11}'))\geq 6$, $\Omega=4$ and  Lemma \ref{l3}, $\left| E_{g'=(0,1)}(v_{1}'v_{2}'\dots v_{11}')\right| \geq 2$. By Lemma 
 \ref{l3}, $g'(v_{10}'v_{11}')\ne (0,1)$, $g'(v_{8}'v_{9}')= (0,1)$ and $g'(v_{10}'v_{11}')= (1,0)$. Since $\omega(v_{2}'v_{3}')+\omega(v_{4}'v_{5}')\geq 4$ and Lemma \ref{l3}, $\omega(v_{10}'v_{11}')=4$. By Lemma \ref{l43}, $\omega(v_{8}'v_{9}')=\omega(v_{6}'v_{7}')=1$, which is a contradiction as $\omega(E_{g'=(0,0)}(v_{1}'v_{2}'\dots v_{11}'))\geq 6$. Thus, $t'\geq 12$. Similarily, $t''\geq 12$, which is a contradiction as Remark \ref{re1}. 

If $t=8$, then $\omega(E_{g'=(0,0)}(v_{0}'v_{1}'\dots v_{8}'))\geq 5$. Thus $\left| E_{g'=(1,1)}(v_{0}'v_{1}'\dots v_{8}')\right| \geq 2$. Similar to Lemma \ref{l5}, we can prove that $g(v_{2}'v_{3}')=(1,1)$ and $g(v_{6}'v_{7}')=(1,1)$. Thus, $g(v_{4}'v_{5})=(0,1)$. Futher, we have $t'\geq 10$ and $t''\geq 10$,  which is a contradiction as Remark \ref{re1}. The proof of $t\ne9$ is similar to the proof of $t=8$.

If $t=10$,~then $\omega(E_{g'=(0,0)}(v_0'\cdots v_{10}'))\ge6$,~thus $|E_{g'=(1,1)}(v_0'\cdots v_{10}')|\ge 2$.~If $t'=4$,~then $g'(v_2'v_3')=(0,1)$.~If $ g'(v_4'v_5')=g'(v_6'v_7')=g'(v_8'v_9')=(1,1)$,~similar to the proof of Lemma \ref{l5},~a contradiction.~Thus,~there are exactly two edges in $\left\lbrace v_4'v_5',~v_6'v_7',~v_8'v_9'\right\rbrace\cap E_{g'=(1,1)}(C) $,~neither of the weight of them are $1$ as $\Omega\le4$ and $\omega(E_{g'=(0,0)}(v_0'\cdots v_{10}'))=\omega(E_{g'=(1,1)}(v_0'\cdots v_{10}'))$,~thus $\Omega=2$,~which is a contradiction as Lemma $\ref{l3}$.~Thus $t'\ge 7$,~if $t'=7$,~then $\Omega\ne 2$ and $v_2v_3\in\left\lbrace v_1'v_2',v_3'v_4',v_5'v_6'\right\rbrace $ by Lemma \ref{l3},~Lemma $\ref{l42}$ and Lemma $\ref{l43}$.~By Lemma \ref{l3} and \ref{l42} unique edge of $\left\lbrace v_1'v_2',v_3'v_4',v_5'v_6'\right\rbrace \cap E_{g'=(1,1)}$ is weighted $1$,~a contadiction as Lemma \ref{l3}.~Therefore,~$t'\ge 12$.~Similarily,~$t''\ge12$,~a contradiction as Remark \ref{re1}.

The proof of $t\ne11$ is similar to $t\ne10$.

\begin{stepp}\label{step3}
	$\left|E(S')\right|\ne 13$.
\end{stepp}

\rm 	Suppose to be contrary that $|E(S')|=13$. By symmetry,~assume there is an edge in $ \left\lbrace v_1'v_2',v_3'v_4',v_5'v_6',v_7'v_8'\right\rbrace $ weighted $2$ which is $v_2v_3$ if $\Omega\ne2$. Since $\omega(E_{g'=(0,0)}(C))\leq 8$ and $\left| E_{g'=(0,0)}(S) \right|=7$, there is only one edge in $E(S)\cup E_{g'=(0,0)}(C)$ weighted $2$. If $\Omega=4$, combining with Figure \ref{fig1}, then $\omega(v_7v_8)\ne 2$.

\bf Case 3.1 \rm $t=5$.

Since $t=5$,~we have $g'(v_2'v_3')=(1,1)$ and $g'(v_4'v_5')=(0,1)$ by Lemma \ref{l3},~then $t'\ge7$ and $t'\ne8.$

If $t'=7$,~then $g'(v_6'v_7')=(1,0)$,~$\omega(v_6'v_7')=\omega(v_2'v_3')\ge 2$ and $ \omega(v_4'v_5')\ge 3$ by Lemma \ref{l3},~then $\Omega\ge3.$ By Lemma \ref{l42} and By Lemma \ref{l43} $\omega(v_6'v_7')=\omega(v_2'v_3')\ne 1 $,~then $\omega(v_4'v_5')\ne\Omega$.~Futhermore $\Omega=4$~and~$ \omega(v_4'v_5')=3$. Combining with Lemma \ref{l43},~we have $\omega(v_7'v_8')=2$ and $\omega(v_6'v_7')=1$,~a contradiction as $\omega(v_6'v_7')\ne 1$.

If $t'=9$,~combining with $\omega(E_{g'=(0,0)}(v_1'\cdots v_9'))=\omega(E_{g'=(0,1)}(v_1'\cdots v_9'))\geq 5$, then $\left| E_{g'=(0,1)}(v_1'\cdots v_9')\right| \geq 2$. By Lemma \ref{l3}, $g'(v_8'v_9')\ne (0,1)$. Thus, $g'(v_6'v_7')=g'(v_8'v_9')=(1,0)$. Since $\omega(E_{g'=(0,1)}(v_1'\cdots v_9'))\geq 5$ and $\left| E_{g'=(0,1)}(v_1'\cdots v_9')\right|=2$, $\Omega\geq 3$. By Lemma \ref{l42} and By Lemma \ref{l43}, $\omega(v_4'v_5')\ne\Omega$ and $\omega(v_6'v_7')\ne\Omega$. Futher, $\Omega=4$, $\omega(v_4'v_5')\ne1$ and $\omega(v_6'v_7')\ne1$. By Lemma \ref{l43},  and ,~a contradiction as $\omega(v_2'v_3')\ne 1$.

If $t'=10$,~then $\left\lbrace v_6'v_7',v_8'v_9'\right\rbrace=(1,0),(0,1)$. if one edge in $\left\lbrace v_4'v_5',v_6'v_7',v_8'v_9'\right\rbrace $ has a weight of $4$, then the remaining two edge weights are 1, which is a contradiction as Lemma \ref{l3}. Thus, $\omega(E_{g'=(0,1)}(v_1'\cdots v_10'))=6$ and $\Omega\geq 3$. By Lemma \ref{l42} and By Lemma \ref{l43},$\omega(v_1'v_2')=2$. Futher, $\omega(v_2'v_3')=1$, which is a contradiction as $\omega(v_2'v_3')=\omega(v_1'v_2')+\omega(v_3'v_4')\geq 3$.

If $t'=11$,~then there is no edge in $\left\lbrace v_2'v_3',v_4'v_5',v_6'v_7',v_8'v_9',v_{10}'v_{11}'\right\rbrace $ weighted $4$ and there is no edge in $\left\lbrace v_1'v_2',v_3'v_4'\right\rbrace $ weighted $2$. Since there is only one edge in $E(S)\cup E_{g'=(0,0)}(C)$ weighted $2$ and Lemma \ref{l42}, combining $\omega(v_2'v_3')\ne 1$,~$\Omega\ne 4$. Since $\omega(E_{g'=(0,0)}(v_1'\cdots v_{11}'))\geq 6$ and Lemma $\ref{l43}$, there is no edge in $\left\lbrace v_4'v_5',v_6'v_7',v_8'v_9',v_{10}'v_{11}'\right\rbrace $ weighted $3$. Thus, $\Omega\ne 3$. Futher, $\omega(v_2'v_3')=\omega(v_4'v_5')=\omega(v_6'v_7')'=\omega(v_8'v_9')=\omega(v_{10}'v_{11}')=2$ and $g'(v_4'v_5')=g'(v_6'v_7')=g'(v_8'v_9')=(0,1)$,~a contradiction as Lemma \ref{l41}.

If $t'=12$,~then there is no edge in $\left\lbrace v_2'v_3',v_4'v_5',v_6'v_7',v_8'v_9',v_{10}'v_{11}'\right\rbrace $ weighted $4$ and there is no edge in $\left\lbrace v_1'v_2',v_3'v_4'\right\rbrace $ weighted $2$ by there is only one edge in $E(S)\cup E_{g'=(0,0)}(C)$ weighted $2$. Combining $\omega(v_2'v_3')\ne 1$,~$\Omega\ne 4$. Since $\omega(E_{g'=(0,0)}(v_1'\cdots v_{12}'))\geq 7$ and Lemma $\ref{l43}$, there is no edge in $\left\lbrace v_4'v_5',v_6'v_7',v_8'v_9',v_{10}'v_{11}'\right\rbrace $ weighted $3$. Then $\Omega\ne 3$,~a contradiction as Lemma $\ref{l3}$.

If $t'=13$,~then there is no edge in $\left\lbrace v_2'v_3',v_4'v_5',v_6'v_7',v_8'v_9',v_{10}'v_{11}'\right\rbrace $ weighted $4$ and there is no edge in $\left\lbrace v_1'v_2',v_3'v_4'\right\rbrace $ weighted $2$ by there is only one edge in $E(S)\cup E_{g'=(0,0)}(C)$ weighted $2$. Combining $\omega(v_2'v_3')\ne 1$,~$\Omega\ne 4$. Since $\omega(E_{g'=(0,0)}(v_1'\cdots v_{12}'))\geq 7$ and Lemma $\ref{l43}$, there is no edge in $\left\lbrace v_6'v_7',v_8'v_9'\right\rbrace $ weighted $3$. If $\Omega=3$, then one of the weights of $v_{4}'v_{5}'$ and $v_{10}'v_{11}'$ is $3$. If $\omega(v_{4}'v_{5}')=3$, then $\omega(v_6'v_7')=\omega(v_8'v_9')=1$ and $\omega(v_7'v_8')=2$ by Lemma \ref{l413}. Since $\omega(E_{g'=(0,0)}(v_1'\cdots v_{12}'))\geq 7$, $g'(v_6'v_7')=g'(v_8'v_9')=g'(v_{10}'v_{11}')=(0,1)$ and $\omega(v_{10}'v_{11}')\geq 2$. If $\omega(v_{10}'v_{11}')=2$, then $\omega(v_2'v_3')=2$. Consider the maximal $E_{g''=(0,0)}(G)$-alternating $S_0$ containing $v_{4}'v_{5}'$, $v_6'v_7'$, $v_8'v_9'$ and $v_{10}'v_{11}'$. Thus, $\left| E(S_0)\right| \leq 7$, which is a contradiction as Lemma \ref{l41}. Thus, $\omega(v_{10}'v_{11}')=3$. Similarily, if $\omega(v_{10}'v_{11}')=3$, then we can also have $\omega(v_{4}'v_{5}')=3$ and $\omega(v_6'v_7')=\omega(v_8'v_9')=1$. Consider $P_{(1,0),(0,1)}$, $v_{10}$ be a vertex with odd degree in $P_{(1,0),(0,1)}$ and another vertex with odd degree in the same component is not in $S$ by Lemma \ref{l3},~a contradiction as Lemma $\ref{l6}$. Thus, $\Omega=2$. Since $\omega(E_{g'=(0,0)}(v_1'\cdots v_{12}'))\geq 7$ and $\Omega=2$, $\left| E_{g'=(0,0)}(v_1'\cdots v_{12}')\right| =4$. Consider the maximal $E_{g'''=(0,0)}(G)$-alternating $S_1$ containing $E_{g'=(0,0)}(v_1'\cdots v_{12}')$, thus $\left| E(S_1)\right| \leq 7$, which is a contradiction as Lemma \ref{l41}.
Therefore, $t'\geq 14$. Let us consider $t''$.

If $t''=7$,~then $g'(v_6'v_7')=(1,0)$. Since $\omega(E_{g'=(0,0)}(v_0'\cdots v_{5}'))=\omega(E_{g'=(1,1)}(v_0'\cdots v_{5}'))$ and $\omega(E_{g'=(0,1)}(v_0'\cdots v_{7}'))=\omega(E_{g'=(1,1)}(v_0'\cdots v_{7}'))$, $\omega(v_0'v_1')=\omega(v_4'v_5')=\omega(v_0'v_1')+\omega(v_2'v_3')$,~a contradiction.

If $t''=9$,~then $g'(v_6'v_7')=(1,0)$ and $g'(v_8'v_9')\ne (1,0)$ by Lemma \ref{l3}. Since $\omega(E_{g'=(0,0)}(v_0'\cdots v_{9}'))=\omega(E_{g'=(1,1)}(v_0'\cdots v_{9}'))\geq 5$, $g'(v_6'v_7')=g'(v_8'v_9')=(1,0)$, which is a contradiction.

If $t''=10$,~then $g'(v_6'v_7')=g'(v_8'v_9')=(1,0)$~by Lemma \ref{l3}. Since $\omega(E_{g'=(1,0)}(v_0'\cdots v_{5}'))=\omega(E_{g'=(0,1)}(v_0'\cdots v_{5}'))$ and $\omega(E_{g'=(0,1)}(v_0'\cdots v_{10}'))=\omega(E_{g'=(1,1)}(v_0'\cdots v_{10}'))$, $\omega(v_0'v_1')=\omega(v_4'v_5')=\omega(v_0'v_1')+\omega(v_2'v_3')$,~a contradiction.

If $t''=11$,~then $g'(v_6'v_7')=g'(v_8'v_9')=(1,0)$,~$g'(v_{10}'v_{11}')=(0,1)$~by Lemma \ref{l3}. We can easily check that  there is no edge in $\left\lbrace v_2'v_3',v_6'v_7',v_8'v_9',v_{10}'v_{11}'\right\rbrace $ weighted $ 1$.~Therefore,~$\Omega=2$ by Lemma \ref*{l42} and Lemma \ref*{l43}. Thus, $\omega(v_6'v_7')+\omega(v_8'v_{9}')\leq 4$, which is a contradiction as $\omega(E_{g'=(0,0)}(v_0'\cdots v_{11}'))=\omega(v_6'v_7')+\omega(v_8'v_{9}')\geq 6$.

If $t''=12$,~then $ |\left\lbrace v_6'v_7',v_8'v_9',v_{10}'v_{11}'\right\rbrace \cap E_{g'=(1,0)}|=2$ and  $|\left\lbrace v_6'v_7',v_8'v_9',v_{10}'v_{11}'\right\rbrace \cap E_{g'=(0,1)}|=1$.~We can easily check that  there is no edge in $\left\lbrace v_2'v_3',v_6'v_7',v_8'v_9',v_{10}'v_{11}'\right\rbrace $ weighted $ 1$.~Therefore,~$\Omega=2$ by Lemma \ref*{l42} and Lemma \ref*{l43},~Thus, $\omega(E_{g'=(1,0)}\\(v_0'\cdots v_{12}'))\le 4$,~which is a contradiction as $\omega(E_{g'=(1,0)}(v_0'\cdots v_{12}')=E_{g'=(0,0)}\\(v_0'\cdots v_{12}')\ge7.$

If $t''=13$,~then there is no edge in $\left\lbrace v_2'v_3',v_4'v_5',v_6'v_7',v_8'v_9',v_{10}'v_{11}',v_{12}'v_{13}'\right\rbrace $ weighted $4$ by Lemma \ref{l42},~futhermore $\Omega\ne 4$. Thus,~ $\Omega=3$.~Since $\omega(E_{g'=(1,0)})\\(v_0'\cdots v_{13}')=\omega(E_{g'=(0,0)})(v_0'\cdots v_{13}')=7$.~Consider the maximal $E_{g''=(0,0)}$-alternating segment~$S^*$ containing $v_6'v_7'$,~$v_8'v_9'$ and $v_{10}'v_{11}'$ in $E_{g''=(0,0)}(S^*)$.~Thus $|E(S^*)|\le7$,~a contradiction as Lemma \ref{l41}.

Therefore $t''\ge14$,~a contradiction as \ref{l5}.

\bf Case 3.2 \rm $t=7$.

By Lemma $\ref{l5}$,~$g'(v_2'v_3')=g'(v_4'v_5')=(1,1)$~and $g'(v_6'v_7')=(0,1)$. Then $t'\ge 11$ by Lemma \ref{l3}.

If $t'=11$,~then $g'(v_6'v_7')=g'(v_8'v_9')=(0,1)$,~$g'(v_{10}'v_{11}')=(1,0)$~and $\omega(v_6'v_7')+\omega(v_8'v_9')=6$.~Futhermore,~by Lemma \ref{l41},~we have $\Omega=4$ and $\omega(v_6'v_7')=\omega(v_8'v_9')=3$ ,~then $\omega(v_1'v_2')=2$ and $\omega(v_0'v_1')=1$,~which is a contradiction as $\omega(v_0'v_1')=\omega(v_6'v_7')$ by Lemma \ref{l3}.

If $t'=12$,~the one of  $g'(v_6'v_7'),g'(v_8'v_9')$,~$g'(v_{10}'v_{11}')$ is ~$(1,0)$ and the other two are $(0,1)$.~Since these total weight is $7$,~one of which is weighted $4$.~By Lemma \ref{l42},~the other one is weighted $1$,~a contradiction. 

If $t'=13$,~then $g'(v_8'v_9')\ne(0,1)$,~$g'(v_{12}'v_{13}')\ne(0,1) $~$g'(v_{10}'v_{11}')=(0,1)$ and $\omega(v_6'v_7')+\omega(v_{10}'v_{11}')=7$ by Lemma \ref{l3} and Lemma \ref{l41}.~Since these total weight is $7$,~one of which is weighted $4$.~By Lemma \ref{l42},~the other one is weighted $1$,~a contradiction. 

Finally, we have $t'\ge14$,~then $\omega(v_7v_8)\ne 2$ if $\Omega=4$,~$t''\ge 13$ by Lemma \ref{l3}.~

If $t''=13$,~then $g'(v_8'v_9')=g'(v_{12}'v_{13}')=(1,0)$,~$g'(v_{10}'v_{11}')=(0,1) $~ and $\omega(v_8'v_9')+\omega(v_{12}'v_{13}')=7$ by Lemma \ref{l3} and Lemma \ref{l41}.~Since these total weight is $7$,~one of which is weighted $4$.~By Lemma \ref{l42},~the other one is weighted $1$,~a contradiction. 

Therefore $t''\ge14$,~a contradiction as Lemma \ref{l5}.

\bf Case 3.3 \rm $t=8$.

By the similar proof of Lemma \ref{l5},~we have $g'(v_2'v_3')=g'(v_6'v_7')=(1,1)$ and $g'(v_4'v_5')=(0,1)$.~Then $t'\ge 13 $ by Lemma \ref{l3}.

If $t'=13$,~then $g'(v_8'v_9')=g'(v_{12}'v_{13}')=(1,0)$ and $g'(v_4'v_5')=g'(v_{10}'v_{11}')=(0,1)$ by Lemma \ref{l3} and Lemma \ref{l41}.~Since $\omega(v_4'v_5')\ne 4$ and $\omega(v_{10}'v_{11}')\ne 4$ as there is only one edge in $E(S)\cup E_{g'=(0,0)}(C)$ weighted $2$,~this is a contradiction as Lemma \ref{l3}.

Therefore, $t'\ge14$,~then $t''\ge 13$ by Lemma \ref{l3}.

If $t''=13$,~then $g'(v_4'v_5')=g'(v_{12}'v_{13}')=(0,1)$,~$g'(v_{10}'v_{11}')=g'(v_8'v_9')=(1,0) $.~Since $\omega(v_4'v_5')+\omega(v_{12}'v_{13}')=7$ by Lemma \ref{l3}, $\omega(v_4'v_5')\ne 4$ and $\omega(v_{12}'v_{13}')\ne 4$ as there is only one edge in $E(S)\cup E_{g'=(0,0)}(C)$ weighted $2$,~this is a contradiction as Lemma \ref{l3}.

Therefore $t''\ge14$,~a contradiction as Lemma \ref{l5}.

\bf Case 3.4 \rm $t=9$.

By the similar proof of Lemma \ref{l5},~we have $g'(v_2'v_3')=g'(v_6'v_7')=(1,1)$,~$g'(v_4'v_5')\ne(1,1)$ and $g'(v_8'v_9')\ne(1,1)$.~Then $t'\ge 13 $ by Lemma \ref{l3}.

If $t'=13$,~then $g'(v_{12}'v_{13}')=(1,0)$ and there are exactly two edges in of $\left\lbrace v_4'v_5',v_8'v_9',v_{10}'v_{11}'\right\rbrace  $ in $E_{g'=(0,1)}(C)$ and the total weight is $7$.~Futhermore,~there is an edge in $\left\lbrace v_4'v_5',v_8'v_9',v_{10}'v_{11}'\right\rbrace  $ weighted $4$,~which is a contradiction as there is only one edge in $E(S)\cup E_{g'=(0,0)}(C)$ weighted $2$.

Therefore, $t'\ge14$,~then $t''\ge 13$ by Lemma \ref{l3}.

If $t''=13$,~then $g'(v_{12}'v_{13}')=(0,1)$,~since the only edge which can be weighted $4$ in $S$ is $v_2'v_3'$ which is in $E_{g''=(1,1)}(C)$ by the similar proof of Lemma \ref{l5}.~But the total weight of $E_{g'=(0,0)}(C)$ in $v_0\cdots v_{13}$ is $7$,~a contradiction as Lemma \ref{l3}.

Therefore $t''\ge14$,~a contradiction as Lemma \ref{l5}.

\bf Case 3.5 \rm $t=10$.

If $t'=4$,~then $g'(v_{2}'v_{3}')=(0,1)$ by Lemma \ref{l3},~then there are exactly two edges of $\left\lbrace v_4'v_5',v_6'v_7',v_8'v_9'\right\rbrace \cap E_{g'=(1,1)}$ and total weight is $6$ by the proof similar to Lemma \ref{l5}.~Thus their weight is not  $1$,~then $\Omega=2$.~By Lemma \ref{l3},~$ g'(v_4'v_5')=g'(v_6'v_7')=g'(v_8'v_9')=(1,1)$.~Similar to the proof Lemma \ref{l5},~this is a contradiction.

If $t'=7$,~then there is only one edge in $v_0\cdots \cap E_{g'=(0,1)}$ and it is weighted $4$,~then this edge is $v_2'v_3'$.~By Lemma \ref{l3},~$g'(v_8'v_9')=(1,1)$ is weighted $2$. We can easily find a contradiction by Lemma \ref{l42} and Lemma \ref{l3}.

If  $t'=13$,~then $\Omega=4$ and $g'(v_2'v_3')=(0,1)$ is weighted $4$ since the only edge which can be weighted $4$ in $S$ is $v_2'v_3'$ and the total weght of $E_{g'=(0,0)}(v_1\cdots v_{13}) $ is $7$. We can obtain a contradiction as Lemma \ref{l3} and $g'(v_8'v_9')\ne 4$.

Therefore, $t'\ge14$,~then $t''\ge 13$ by Lemma \ref{l3}.

If $t''=13$,~then the total weight of $E_{g'=(0,0)}(C)$ and $E_{g'=(1,1)}(C)$ are at least $7$,~a contradiction as the only edge which can be weighted in $S$ is $v_2'v_3'$. 

Therefore $t''\ge14$,~a contradiction as Lemma \ref{l5}.

Similarly,~we obtain that $t\ne 11$.

\bf Case 3.6 \rm $t=12$.

By there is only one edge in $E(S)\cup E_{g'=(0,0)}(C)$ weighted $2$ and Lemma \ref{l42},~ $\omega(e)\ne4$ for each edge in $S$.

If $t'=4$ and $g'(v_4'v_5')=g'(v_6'v_7')=g'(v_8'v_9')=g'(v_{10}'v_{11}')=(1,1)$,~then we can proof a contradiction which is similar to Lemma \ref{l5}.~Thus,~by $\omega(E_{g'=(1,1)}(S))=\omega(E_{g=(0,0)}(S))=7$ and $\omega(e)=\ne4$ for each edge in $S$,~we have $\Omega\ge3$.~If $\Omega=4$,~then $ \omega(v_5'v_6')\ne 2$ by Lemma \ref{l42}. Then there is only one edge weighted $2$ in $\left\lbrace v_1'v_2',v_3'v_4'\right\rbrace$  a contradiction as $\omega(E_{g=(0,1)}(v_1v_2v_3))=\omega(E_{g=(0,0)}(v_1v_2v_3))$.~Therefore,~ $\Omega=3$,~then there is an edge in $\left\lbrace v_4'v_5',v_6'v_7',v_8'v_9',v_{10}'v_{11}'\right\rbrace $ weighted $3$.~Then $v_1'$ and $v_3'$ are vertices of odd degree in $P_{(1,0),(0,1)}$ assume ${v_{t_1}}$ be a vertex with odd degree in the same component of $P_{(1,0),(0,1)}$ of $v_1'$ and ${v_{t_3}}$ be a vertex with odd degree in the same component of $P_{(1,0),(0,1)}$ of $v_3'$.~by Lemma \ref{l3},~$t_1\ge 13$ or $t_3\ge 13$,~a contradiction as Lemma \ref{l6}.

If $t'=6$,~then $ g'(v_6'v_7')=g'(v_8'v_9')=g'(v_{10}'v_{11}')=(1,1)$ and $\omega(v_6'v_7')+\omega(v_8'v_9')+\omega(v_{10}'v_{11}')=7$,~by finding a maximal $E_{g''=(0,0)}$-alternating segment of $C$ such $g''(v_6'v_7')=g''(v_8'v_9')=g''(v_{10}'v_{11}')=(0,0)$,~then the length of this segment is at most $7$,~a contradiction as Lemma \ref{l41}.

Thus $t'\ge 9$.~If $t'=9$,~then there are exactly two edges of in $E_{g'=(1,1)}(v_0'\cdots v_{12}')$.
By Lemma \ref{l3},~$g'(v_2'v_3')=(1,1)$ is weghted $4$,~a contradiction as the total weight of $E_{g'=(1,1)}(C)$ is at least $9$.

Similar to $t'\ne 9$,~we have $t'\ne 10$,~thus $t'\ge 14$ by Lemma \ref{l3} and Lemma $\ref{l42}$.

Then $t''\ge 14$,~a contradiction as Lemma \ref{l5}.

Similarly,~we obtain that $t\ne 13$.
 Therefore,~$|E(S')|\ne 13$ by Lemma \ref{l41}.
 
By $|E(S')|$ is odd and $9\le|E(S')|\le13$,~a contradiction as $|E(S')|\notin \left\lbrace 9,11,13\right\rbrace $.
$\hfill\qed$



\end{document}